\documentclass[12pt]{amsart}
\usepackage{graphicx} % Required for inserting images
\usepackage[utf8]{inputenc}
\usepackage{amsmath,amssymb,amsfonts,amsthm}
\usepackage{hyperref}

\newcommand{\loo}{\ensuremath{\mathcal{L}_{\omega_1, \omega}}}

\newtheorem{thm}{Theorem}[section]
\newtheorem{lem}[thm]{Lemma}
\newtheorem{prop}[thm]{Proposition}
\newtheorem{defin}[thm]{Definition}
\newtheorem{cor}[thm]{Corollary}
\newtheorem{que}[thm]{Open Question}

\begin{document}

\title{Shelah's conjecture fails for higher cardinalities}

\author{Georgios Marangelis}
\email{geormara@math.auth.gr}
\address{Department of Mathematics, Aristotle University of Thessaloniki, Greece}
 

\thanks{This paper is part of the author's Ph.D. Thesis at the Department of Mathematics of Aristotle University in Thessaloniki, Greece, under the guidance of Dr. Ioannis Souldatos.}

\begin{abstract}
    
The main goal of this paper is to generalize the results that where presented in \cite{Kurepatrees} for $\aleph_1$-Kurepa trees to $\aleph_{\alpha+1}$-Kurepa trees. 

We construct an $\mathcal{L}_{\omega_1, \omega}$-sentence $\psi_{\alpha}$, that codes $\aleph_{\alpha + 1}$-Kurepa trees, for some countable $\alpha$. One of the main results for its spectrum (the spectrum of a sentence is the class of all cardinals for which there exists some model of the sentence) is the following:

\vspace{10pt}
    \textit{It is consistent that $2^{\aleph_{\alpha}} < 2^{\aleph_{\alpha + 1}}$, that $2^{\aleph_{\alpha + 1}}$ is weakly inaccessible and that the spectrum of $\psi_{\alpha}$ is equal to $[\aleph_{0}, 2^{\aleph_{\alpha + 1}})$. }
\vspace{10pt}

    This relates to a conjecture of Shelah, that if $\aleph_{\omega_1}<2^{\aleph_0}$ and there is a model of some \loo-sentence of size $\aleph_{\omega_1}$, then there is a model of size $2^{\aleph_{0}}$. Shelah calls $\aleph_{\omega_1}$ the local Hanf number below $2^{\aleph_0}$ and proves the consistency of his conjecture in \cite{19}. 
    It is open if the negation of Shelah's conjecture is consistent. Our result proves that if we replace $2^{\aleph_0}$ by $2^{\aleph_{\alpha + 1}}$, it is consistent that there is no local Hanf number.
    
    There are some interesting results for the amalgamation spectrum too (the amalgamation spectrum is defined similarly to the spectrum, but we require that $\kappa$-amalgamation holds plus there is at least one model of size $\kappa$). We prove that $\kappa$-amalgamation for \loo- sentences is not absolute. More specifically we prove for $\alpha > 0$ finite, it is consistent that:
    
   \begin{itemize}
    \item $2^{\aleph_{\alpha}} = \aleph_{\alpha + 1} < \lambda \leq 2^{\aleph_{\alpha + 1}}, cf(\lambda) > \aleph_{\alpha}$ and $AP-Spec(\psi_{\alpha})$ contains the whole interval $[\aleph_{\alpha + 2}, \lambda]$ and possibly $\aleph_{\alpha + 1}$.

    \item $2^{\aleph_{\alpha}} = \aleph_{\alpha + 1} < 2^{\aleph_{\alpha + 1}}$, $2^{\aleph_{\alpha + 1}}$ is weakly inaccessible and $AP-Spec(\psi_{\alpha})$ contains the whole interval $[\aleph_{\alpha + 2}, 2^{\aleph_{\alpha + 1}})$ and possibly $\aleph_{\alpha + 1}$.
\end{itemize}
\end{abstract}

\maketitle

\section{Introduction}

\begin{defin}
    For an $\mathcal{L}_{\omega_1, \omega}$ sentence $\phi$, the $\mathbf{spectrum}$ of $\phi$ is the class

    \begin{center}
        Spec($\phi$) = $\{ \kappa | \exists M \models \phi$ and $|M| = \kappa \}$.
    \end{center}

    If Spec($\phi$) = $[\aleph_{0}, \kappa]$, we say that $\phi$ characterizes $\kappa$.

    The $\mathbf{maximal}$ $\mathbf{models}$ $\mathbf{spectrum}$ of $\phi$ is the class

    \begin{center}
        MM-Spec($\phi$) = $\{ \kappa | \exists M \models \phi$ and $|M| = \kappa$ and M is maximal $\}$.
    \end{center}

    We can, also, define the $\mathbf{amalgamation}$ $\mathbf{spectrum}$ of $\phi$, AP-Spec($\phi$) and the $\mathbf{joint}$ $\mathbf{embedding}$ $\mathbf{spectrum}$ of $\phi$, JEP-Spec($\phi$) as follows:

    \begin{center}
        AP-Spec($\phi$) = $\{ \kappa | \phi$ has at least one model of size $\kappa$ and the models of size $\kappa$ satisfy the amalgamation property $\}$\\
        JEP-Spec($\phi$) = $\{ \kappa | \phi$ has at least one model of size $\kappa$ and the models of size $\kappa$ satisfy the joint embedding property $\}$.
    \end{center}
\end{defin}

One of the main notions we meet in this paper is $\mathbf{\kappa}$-$\mathbf{Kurepa}$ $\mathbf{trees}$, so it is useful to give a definition from the beginning.

\begin{defin}\label{Kurepa}
    Assume $\kappa$ is an infinite cardinal. A $\mathbf{\kappa}$-$\mathbf{tree}$ has height $\kappa$ and each level has at $< \kappa$ elements. A $\mathbf{\kappa}$-$\mathbf{Kurepa}$ $\mathbf{tree}$ is a $\kappa$-tree with at least $\kappa^{+}$ many branches of height $\kappa$.
\end{defin}

In this paper we generalize some results, that have already been presented in \cite{Kurepatrees}, from $\aleph_{1}$ to $\aleph_{\alpha + 1}$, where $\alpha$ is countable.

In \cite{MR4292936, MR4323604}, there are some interesting characterizations for the ``branch spectrum of Kurepa trees" , i.e. the set \[Sp_{\omega_{1}} = \{ \lambda | \text{ there exists a Kurepa tree with $\lambda$ many cofinal branches}\}.\] More specifically M. Poor and S. Shelah, define a set, $S$, of ordinals that satisfies certain set-theoretic assumptions and closure properties
and by assuming Continuum Hypothesis (CH) and $2^{\omega_{1}} = \omega_{2}$, they find a forcing extension in which $S = \{ \alpha |$ there exists a Kurepa tree with $\omega_{\alpha}$ many cofinal branches$\}$.

The notion of the ``branch spectrum of Kurepa trees" is related to the ``spectrum" notion that is defined in \cite{Kurepatrees} (and also in this paper), with the difference that some trees in \cite{Kurepatrees} may not contain all their cofinal branches, or they may not be even be Kurepa tree, while in \cite{MR4292936, MR4323604} the trees always contain all their maximal branches.
It is an open question if the results in \cite{MR4292936, MR4323604}
could be generalized to $\kappa$-Kurepa trees, where $\kappa > \aleph_{1}$.

It is known that the existence of a $\kappa$-Kurepa tree is independent from ZFC. Our purpose is to construct an $\mathcal{L}_{\omega_1, \omega}$-sentence, $\psi_{\alpha}$, for which every $\aleph_{\alpha + 1}$-Kurepa tree gives a model of $\psi_{\alpha}$. 
One of our motivations is Shelah's conjecture about the local Hanf number below $2^{\aleph_{0}}$. Shelah proves in \cite{19} that consistently the local Hanf number below $2^{\aleph_{0}}$ is $\aleph_{\omega_1}$. Despite a lot of effort by various researchers, it remains an open question if it is consistent that there is no local Hanf number below the continuum. 

Shelah and D. Ulrich generalized Shelah's conjecture to higher cardinals and proved that consistency there are local Hanf numbers for cardinals higher than the continuum. More specifically, they prove that it is consistent that for any countable $\alpha > 0$, there is a Hanf number below $\beth_{\alpha + 1}$, which is less than $\aleph_{\beth_{\alpha}^{++}}$.

Their result remains unpublished (so far) and it was a personal communication. In Theorem \ref{consistency} we prove that it is consistent that there is no Hanf number below $2^{\aleph_{\alpha + 1}}$ for every countable $\alpha$. More specifically, we prove that:
\begin{center}
$2^{\aleph_{\alpha}} < 2^{\aleph_{\alpha + 1}}$, $2^{\aleph_{\alpha + 1}}$ is weakly inaccessible and $Spec(\psi_{\alpha}) = [\aleph_{0}, 2^{\aleph_{\alpha + 1}})$.
\end{center}
Our sentence does not refute the initial conjecture of Shelah, since it always has a model of size $2^{\aleph_{0}}$.

\begin{que}
    Is the negation of Shelah's conjecture consistent with ZFC?
\end{que}

In Section 5 (Theorem \ref{spectrum})
we will see the properties that characterize the spectrum of our sentence and in Section 6 we prove that it is consistent with ZFC that Spec($\psi_{\alpha}$) = $[\aleph_{0}, \aleph_{\omega_{\alpha + 1}}]$, or $[\aleph_{0}, 2^{\aleph_{\alpha + 1}})$ and $2^{\aleph_{\alpha + 1}}$ is weakly inaccessible. The interesting part here is that we have found an $\mathcal{L}_{\omega_1, \omega}$-sentence that its spectrum can consistently be right-open and right-closed.

Next, we examine the properties of the maximal model spectrum of our sentence and we prove that we can have maximal models in finite, countable or uncountable many cardinalities. The exact same phenomenon was also noticed in \cite{Kurepatrees}.

We prove some interesting results for the amalgamation spectrum too. Our results come after some similar results in \cite{18, Kurepatrees}, namely that for certain $\alpha$,  $\aleph_{\alpha}$-amalgamation under substructure, is non-absolute. More specifically, in \cite{18}, they prove that $\aleph_{\alpha}$-amalgamation is not absolute for $1 < \alpha < \omega$ and in \cite{Kurepatrees} that $\aleph_{\alpha}$-amalgamation is not absolute (by manipulating the size of $2^{\aleph_{1}}$) for $\aleph_{1} < \aleph_{\alpha} < 2^{\aleph_{1}}$. As regards to our results, in Theorem \ref{JEPAP}, we prove that if $\alpha$ is infinite the amalgamation spectrum is empty, but if $\alpha$ is finite then the amalgamation spectrum contains all the cardinals which are above $2^{\aleph_{\alpha}}$ and belong to the model-existence spectrum,  as well all the cardinals $\aleph_{\alpha + 1} \leq \kappa \leq 2^{\aleph_{\alpha}}$ for which a $(\mu, \aleph_{\alpha}, \kappa)$-Kurepa tree exists, for some $\mu \leq \aleph_{\alpha}$ (see definition \ref{weirdKurepa}). This gives us great leverage in controlling the amalgamation spectrum by manipulating the existence of Kurepa trees. In Corollary \ref{results1}, we can see that $\kappa$-amalgamation for $\mathcal{L}_{\omega_1, \omega}$-sentences is non-absolute, for every $\kappa > 2^{\aleph_{\alpha}}$ that belongs to the Spectrum of $\psi_{\alpha}$. More specifically we prove:

\begin{itemize}
    \item for $\alpha > 0$ finite, it is consistent that $2^{\aleph_{\alpha}} = \aleph_{\alpha + 1} < \lambda \leq 2^{\aleph_{\alpha + 1}}$, with $cf(\lambda) > \aleph_{\alpha}$ and $AP-Spec(\psi_{\alpha})$ contains the whole interval $[\aleph_{\alpha + 2}, \lambda]$ and possibly $\aleph_{\alpha + 1}$
    \item for $\alpha > 0$ finite, it is consistent that $2^{\aleph_{\alpha}} = \aleph_{\alpha + 1} < 2^{\aleph_{\alpha + 1}}$, $2^{\aleph_{\alpha + 1}}$ is weakly inaccessible and $AP-Spec(\psi_{\alpha})$ contains the whole interval $[\aleph_{\alpha + 2}, 2^{\aleph_{\alpha + 1}})$ and possibly $\aleph_{\alpha + 1}$.
\end{itemize}

Finally, we notice that the question for $\aleph_{1}$-amalgamation absoluteness remains open, since the amalgamation-spectrum of our sentence never contains $\aleph_{1}$ (in contrast to the results in \cite{Kurepatrees}, where $\aleph_{1}$ was always in the spectrum).

\begin{que}
     Is $\aleph_{1}$-amalgamation for $\mathcal{L}_{\omega_1, \omega}$-sentences absolute for models of ZFC?
\end{que}

Summarizing the results, in section 2, we present some definitions and some known results for $\kappa$-Kurepa trees and introduce the new notion of $(\mu, \lambda, \kappa)$-Kurepa trees (see definition \ref{weirdKurepa}). Section 3 contains the main construction of our sentence, while in section 4, I prove that the class of the models of this sentence is an Abstract Elementary Class with countable Lowenheim-Skolem number. Section 5 contains the main results concerning the Spectrum, the Maximal Models Spectrum and the Amalgamation Spectrum of the sentence. Finally, section 6 contains some consistency results that give us useful applications of the Spectrum Theorems.

\section{Kurepa trees}

\begin{defin}\label{dfKurepa}
    Assume $\kappa$ is an infinite cardinal and $\lambda \geq \kappa^{+}$. A $\mathbf{(\kappa, \lambda)}$-$\mathbf{Kurepa}$ $\mathbf{tree}$ is a $\kappa$-Kurepa tree with exactly $\lambda$ branches of height $\kappa$. 
    
    KH($\kappa, \lambda$) is the statement that there exists a $(\kappa, \lambda)$-Kurepa tree.

    Define $\mathcal{B}(\kappa) =$ sup$\{ \lambda | KH(\kappa, \lambda)$ holds $\}$.

    A $\mathbf{weak}$ $\mathbf{\kappa}$-$\mathbf{Kurepa}$ $\mathbf{tree}$ is a tree of height $\kappa$, each level has size $\leq \kappa$ elements and there are at least $\kappa^+$ cofinal branches.
\end{defin}

\underline{Comment:}For this paper we will assume that $\kappa$-Kurepa trees are pruned, i.e. every node is contained in a maximal branch of order type $\kappa$.

The following Lemma is useful for some consistency results we prove later. Its proof is already given in \cite{Kurepatrees}.

\begin{lem}\label{B(k)}
    If $\mathcal{B}(\kappa)$ is not maximum, then $cf(\mathcal{B}(\kappa)) \geq \kappa^{+}$.
\end{lem}

Note here that $\kappa$-Kurepa trees are defined for every infinite $\kappa$, but we are only interested when $\kappa$ is uncountable. In the countable case is easy to see, that we always can find $\aleph_{0}$-Kurepa trees (e.g. we can take $(2^{\leq \omega}, \subset)$). So, throughout the rest of the paper when we refer to $\kappa$-Kurepa trees we will assume that $\kappa$ is uncountable.

It is known that the existence of $\kappa$-Kurepa trees is independent from ZFC (see \cite{IndependenceKurepaWeakKurepa, IndependenceKurepa}). It is also known that the existence of weak $\kappa$-Kurepa trees is independent from ZFC (see \cite{IndependenceKurepaWeakKurepa, IndependenceWeakKurepa}).

One other interesting independence result for $\kappa$-Kurepa trees is that it is consistent to have $\kappa$-Kurepa trees for some $\kappa$, but at the same time no $\lambda$-Kurepa trees for some $\lambda \neq \kappa$. In \cite{IndependenceWeakKurepa}, they find a model that has no $\aleph_{1}$-Kurepa trees, but there exist $\aleph_{2}$-Kurepa trees.

\begin{defin}\label{weirdKurepa}
    A $(\mu, \lambda, \kappa)$-Kurepa tree, where $\kappa \geq \lambda \geq \mu$ and $\kappa > \mu$, is a tree of height $\mu$, each level has at $< \lambda$ elements, and there are $\kappa$ branches of height $\mu$. A $(\mu, \mu, \kappa)$-Kurepa tree, is a $\mu$-Kurepa tree with $\kappa$ cofinal branches and a $(\mu, \mu^{+}, \kappa)$-Kurepa tree is a weak $\mu$-Kurepa tree (again with $\kappa$ cofinal branches).
\end{defin}

In the literature, although, there are plenty of results for $\mu$-Kurepa trees and weak $\mu$-Kurepa trees, there are no results for $(\mu, \lambda, \kappa)$-Kurepa trees, when $\lambda > \mu^{+}$.

It is easy to see that for any $\lambda \geq \mu$, it is consistent that there exist such trees. For example, if we assume $2^{< \mu} < \lambda$ and $2^{\mu} > \lambda$, then we see that $(2^{\leq \mu}, \subset)$ is a $(\mu, \lambda, 2^{\mu})$-Kurepa tree.

There are also some easy cases to see that the negation of the existence of such trees is consistent. If $\lambda = \mu$ or $\lambda = \mu^{+}$, then we know that from $\mu$-Kurepa trees and weak $\mu$-Kurepa trees correspondingly. If $\lambda > \mu^{+}$, and $\lambda = \sigma^{++}$ for some $\sigma$, then if we assume the GCH, there are no $(\mu, \lambda, \kappa)$-Kurepa trees. Indeed, the cofinal branches of such trees are at most $(\sigma^{+})^{\mu}$. Then we have that $\sigma^{+} \leq (\sigma^{+})^{\mu} \leq (\sigma^{+})^{\sigma} = 2^{\sigma} = \sigma^{+} < \lambda$. But by definition such a tree should have $\lambda$ cofinal branches, contradiction. 

For the second question, I haven't found results for weak $\kappa$-Kurepa trees at all, so I ask for $\lambda > \kappa$.

\begin{que}

\begin{enumerate}
    \item     Is  the following consistent? There exists a $(\mu, \lambda, \kappa)$-Kurepa tree, for some $\lambda > \mu^{+}$, but there does not exist a $\mu$-Kurepa tree.
    
    \item   Is the existence of $(\mu, \lambda, \kappa)$-Kurepa trees independent for the various values of $\mu$? Here we assume that $\lambda > \mu$.
    
    \item     For which $\mu, \lambda$ is the negation of the existence of $(\mu, \lambda, \kappa)$-Kurepa trees consistent with ZFC? Is the negation of the existence of such trees consistent for all $\mu, \lambda$ at the same time?
    \end{enumerate}
\end{que}

\section{Main construction and first results}

\begin{defin}
    Let $\kappa \leq \lambda$ be infinite cardinals. A sentence $\sigma$ in a language with a unary predicate P admits $(\lambda, \kappa)$, if $\sigma$ has a model M such that $|M| = \lambda$ and $|P^{M}| = \kappa$. In this case, we will say that M is of type $(\lambda, \kappa)$.
\end{defin}

Our goal, now, is to construct an $\mathcal{L}_{\omega_1, \omega}$ sentence such that every $\aleph_{\alpha + 1}$-Kurepa tree (where $\alpha$ is a fixed countable ordinal) belongs to its spectrum.

From \cite{4}, we know the following theorem.

\begin{thm}\label{construction}
    There is a first order sentence $\sigma$ such that for all infinite cardinals $\kappa$, $\sigma$ admits ($\kappa^{++}$, $\kappa$) iff $KH(\kappa^{+}, \kappa^{++})$.
\end{thm}

We will not present the proof for this theorem, but we modify the construction of $\sigma$, in order to construct the desired $\mathcal{L}_{\omega_1, \omega}$ sentence, $\psi_{\alpha}$.

Assume that $\alpha$ is a countable ordinal. The vocabulary $\tau$ consists of the following. It contains the constants 0, $(c_{n})_{n \in \omega}$, the unary symbols $L_{0}, L_{1}, ..., L_{\alpha}, L_{\alpha + 1}$, the binary symbols $S, V, T, <_{1}, <_{2}, ..., <_{\alpha}, <_{\alpha + 1}$ and the ternary symbols $F_{0}, F_{1}, ..., F_{\alpha}, G$. We aim to build an $\aleph_{\alpha + 1}$-Kurepa tree. 

We think of $L_{\alpha + 1}$ to be a set that corresponds to the ``levels" of the tree. We also demand $L_{\alpha + 1}$ to be linearly ordered by $<_{\alpha + 1}$ and 0 to be its minimum element. It is important to note that $L_{\alpha + 1}$ may or may not have a maximum element. Every element $a \in L_{\alpha + 1}$ that is not a maximum element has a successor $b$ that satisfies $S(a, b)$. We will denote the successor of $a$ by $S(a)$. The maximum element, if it exists, we call m and require that is not a successor. For every $a \in L_{\alpha+1}, V(a, \cdot)$ is the set of nodes at level $a$ and we assume that $V(a, \cdot)$ is disjoint from all the $L_0, L_1, \ldots, L_{\alpha + 1}$. If $a\neq b$, then $V(a,\cdot)$ is disjoint from $V(b,\cdot)$. If $V(a, x)$, we will say that $x$ is at the level $a$ and most of the times we will write $x \in V(a)$ instead.

$T$ is a tree ordering on $V = \bigcup_{a \in L_{\alpha+1}} V(a)$. If $T(x, y)$ holds, then $x$ is at some level strictly less than the level of $y$. If $y \in V(a)$ and $b < a$, there is some $x$ such that $x \in V(b)$ and $T(x, y)$. If $a$ is a limit, that is neither a successor nor 0, then two distinct elements in $V(a)$ cannot have the same predecessors. If m is the maximum element of $L_{\alpha + 1}, V(m)$ is the set of maximal branches through the tree. Both ``the height of $T$" and ``the height of $L_{\alpha + 1}$" refer to the order type of $(L_{\alpha + 1}, <_{\alpha + 1})$. We can also stipulate that the %$\aleph_{\alpha + 1}$-Kurepa 
tree is pruned.

Our goal, now is to bound the size of each $L_{\beta}$ by $\aleph_{\beta}$. 
For the first level, we require that $\forall x (L_{0}(x) \leftrightarrow \bigvee_{n} x = c_{n})$. That gives us that $|L_{0}| = \aleph_{0}$.

Each $L_{\beta}, \beta = 1, 2, ..., \alpha$ is linearly ordered by $<_{\beta}$ and may or may not have maximum element. We will see that if $|L_{\beta}| = \aleph_{\beta}$, for $\beta$ successor, then there is no maximum. 

In order to bound the size of $L_{\beta + 1}$ by $\aleph_{\beta + 1}$, we bound the size of each initial segment by $\aleph_{\beta}$. Our treatment is slightly different for $\beta<\alpha$ than for $\beta=\alpha$.
The difference is how we treat the maximum elements.

\underline{Notice:} Defining the $F_{\alpha}(x, \cdot, \cdot)$, we demand that $x$ is not the maximum element of $L_{\alpha + 1}$. We don't have the same restriction for the rest of the $F_{\beta}$'s. That difference plays an important role throughout the rest of the paper.

Let $\beta < \alpha$. For every $x \in L_{\beta + 1}$ there is a surjection $F_{\beta}(x, \cdot, \cdot)$ from $L_{\beta}$ to $(L_{\beta + 1})_{\leq_{(\beta + 1)} x} = \{ b \in L_{\beta + 1} | b \leq_{(\beta + 1)} x \}$. This bounds the size of each initial segment $(L_{\beta + 1})_{\leq_{(\beta + 1)} x}, \beta < \alpha$ by $|L_{\beta}|$. Here, we can observe that if $L_{\beta + 1}$ had a maximum element, then its whole size would be $\leq \aleph_{\beta}$. So, if its size is equal to $\aleph_{\beta + 1}$, it cannot have maximum.

At limit stages we take $L_{\beta}$ as the union of the previous $L_{\gamma}$. The linear order on limit stages is not relevant to the linear orders in the previous stages.

Finally, for every $x \in L_{\alpha + 1}$, \underline{that is not the maximum element}, there is a surjection $F_{\alpha}(x, \cdot, \cdot)$ from $L_{\alpha}$ to $(L_{\alpha + 1})_{\leq_{(\alpha + 1)} x}$ and another surjection $G(x, \cdot, \cdot)$ from $L_{\alpha}$ to $V(x)$. This bounds the size of $(L_{\alpha + 1})_{\leq_{(\alpha + 1)} x}$ and the size of every $V(x)$, which is not maximal level, by $|L_{\alpha}|$.

We have to observe here, that although the upper bound for the size of $L_{\beta}$'s increases as $\beta$ increases, the actual size of $L_{\beta}$'s may decrease, since we just require $F_{\beta}$ to be a surjection.

This construction gives us that for all $\beta = 1, 2, ..., \alpha + 1, |L_{\beta}| \leq \aleph_\beta$ and for all non maximal levels $|V(x)| \leq \aleph_{\alpha}$.

So, our desired $\mathcal{L}_{\omega_1, \omega}$ sentence, $\psi_{\alpha}$ is the conjuction of all the above requirements. 

We want to emphasize here the fact that since well-orderings cannot be characterized by an $\mathcal{L}_{\omega_1, \omega}$ sentence, it is unavoidable that we will be working with non-well ordered trees. However, there is a connection between set-theoretic trees and the models of $\psi_{\alpha}$: For every model of our $\psi_{\alpha}$, if take a cofinal subset of $L_{\alpha + 1}$ with the corresponding levels, we form a tree in the usual set-theoretic notion. 

We can easily see that every $\aleph_{\alpha + 1}$-Kurepa tree gives rise to a model of $\psi_{\alpha}$, but there are models of $\psi_{\alpha}$, that do not code an $\aleph_{\alpha + 1}$-Kurepa tree. For example there is a model of $\psi_{\alpha}$ that codes a tree of height $\omega$, $\omega$- splitting, and which has $2^\omega$ many cofinal branches.
We also observe that every $(\kappa, \lambda, \mu)$-Kurepa tree, $\kappa \leq \aleph_{\alpha + 1}, \lambda \leq \aleph_{\alpha + 1}$ gives rise to a model of $\psi_{\alpha}$.

Our aim is to study the models that code $\aleph_{\alpha + 1}$-Kurepa trees. The dividing line %for models of $\psi_{\alpha}$ to code $\aleph_{\alpha + 1}$-Kurepa trees 
is the size of $L_{\alpha + 1}$. By construction, every initial segment of $L_{\alpha + 1}$ has size at most $\aleph_{\alpha}$. If in addition $|L_{\alpha + 1}| = \aleph_{\alpha + 1}$, then we can embed $\omega_{\alpha + 1}$ cofinally into $L_{\alpha + 1}$. Hence, every model of $\psi_{\alpha}$ of size $\geq \aleph_{\alpha + 2}$ and for which $|L_{\alpha + 1}| = \aleph_{\alpha + 1}$, codes an $\aleph_{\alpha + 1}$-Kurepa tree.

Furthermore,  since $\psi_{\alpha}$ is an $\mathcal{L}_{\omega_{1}, \omega}$ sentence, for every model of $\psi_{\alpha}$ of size $\mu$, we get models  for every cardinality less than $\mu$. These models may or may not code trees (in the set-theoretic sense) and they may or may not contain all their maximal branches. For the rest of the paper, when we refer to a tree, we refer to a model of $\psi_{\alpha}$.

Let $\mathbf{K}_{\alpha}$ be the collection of all models of $\psi_{\alpha}$, equipped with the substructure relation. I.e. for $M, N \in \mathbf{K}_{\alpha}, M \prec_{\mathbf{K}_{\alpha}} N$ if $M \subset N$. Note that $(\mathbf{K}_{\alpha}, \prec_{\mathbf{K}_{\alpha}})$ is an Abstract Elementary Class, which we prove in Theorem \ref{AEC}.

\begin{lem}\label{help}
    If $M \prec_{\mathbf{K}_{\alpha}} N$, then
    \begin{enumerate}
        \item $L_{0}^{M} = L_{0}^{N}$
        \item $L_{1}^{M}$ is initial segment of $L_{1}^{N}$
        \item For $0 \leq \beta \leq \alpha$, if $|L_{\gamma}^{M}| = \aleph_{\gamma}$, for every $\gamma \leq \beta$, then $L_{\beta}^{M} = L_{\beta}^{N}$
        \item For $0 \leq \beta \leq \alpha$, if $|L_{\gamma}^{M}| = \aleph_{\gamma}$, for every $\gamma \leq \beta$, then $L_{\beta + 1}^{M}$ is an initial segment of $L_{\beta + 1}^{N}$
        \item If $|L_{\beta}^{M}| = \aleph_{\beta}$, for every $\beta \leq \alpha$, then $V^{M}(x) = V^{N}(x)$, for every non maximal $x \in L_{\alpha + 1}^{M}$
        \item the tree ordering is preserved
    \end{enumerate}
\end{lem}

\begin{proof}
    (1) It is immediate from the definition

    (2) Towards contradiction, assume that there is a $y \in L_{1}^{M}$ and a $x \in L_{1}^{N} \setminus L_{1}^{M}$ with $x <_{1} y$. Then the function $F_{0}^{M}(y, \cdot, \cdot)$ defined on $M$ disagrees with the function $F_{0}^{N}(y, \cdot, \cdot)$ defined on $N$. Contradiction.

    (3) + (4) We  prove both (3) and (4) by induction on $0 \leq \beta \leq \alpha$. For $\beta = 0$, we have already proved that $L_{0}^{M} = L_{0}^{N}$ and $L_{1}^{M}$ is an initial segment of $L_{1}^{N}$. 
    
    Suppose we have the results for all ordinals $\leq \beta$, i.e. if $|L_{\gamma}^{M}| = \aleph_{\gamma}$, for $\gamma \leq \beta$, then we know that $L_{\beta}^{M} = L_{\beta}^{N}$ and $L_{\beta + 1}^{M}$ is an initial segment of $L_{\beta + 1}^{N}$. Now we have to show the results for $\beta + 1$, i.e. if $|L_{\gamma}^{M}| = \aleph_{\gamma}$, for $\gamma \leq \beta + 1$, then $L_{\beta + 1}^{M} = L_{\beta + 1}^{N}$ and that $L_{\beta + 2}^{M}$ is initial segment of $L_{\beta + 2}^{N}$. 
    \begin{itemize}
        \item Suppose that $L_{\beta + 1}^{M} \neq L_{\beta + 1}^{N}$, then we can find a $y \in L_{\beta + 1}^{N} \setminus L_{\beta + 1}^{M}$, such that $y \geq_{\beta + 1} x$, for all $x \in L_{\beta + 1}^{M}$. But then the function $F_{\beta}^{N}(y, \cdot, \cdot)$ is a surjection from $L_{\beta}^{N}$ to $(L_{\beta + 1}^{N})_{\leq_{(\beta + 1)} y} \supseteq L_{\beta + 1}^{M}$, which contradicts the fact that $|L_{\beta}^{N}| = \aleph_{\beta}$.
        \item Suppose that $L_{\beta + 2}^{M}$ is not an initial segment of $L_{\beta + 2}^{N}$. Then, there is a $y \in L_{\beta + 2}^{M}$ and a $x \in L_{\beta + 2}^{N} \setminus L_{\beta + 2}^{M}$ with $x <_{(\beta + 2)} y$. Then the function $F_{\beta + 1}^{M}(y, \cdot, \cdot)$ defined on $M$ disagrees with the function $F_{\beta + 1}^{N}(y, \cdot, \cdot)$ defined on $N$. Contradiction.
    \end{itemize}    

    In (3), we have to prove the result for limit ordinals. So, for $\beta$ limit, if we have the result for all ordinals $< \beta$, it is immediate that $L_{\beta}^{M} = L_{\beta}^{N}$, since $L_{\beta} = \bigcup_{\gamma < \beta} L_{\gamma}$ and $L_{\gamma}^{M} = L_{\gamma}^{N}$, for every $\gamma < \beta$.
  
    (5) Towards contradiction assume that there is a non maximal $x \in L_{\alpha + 1}^{M}$ and a $y \in V^{N}(x) \setminus V^{M}(x)$. Then, $G^{M}(x, \cdot, \cdot)$ disagrees with $G^{N}(x, \cdot, \cdot)$. Contradiction.

    (6) It is immediate from the definition.
    
\end{proof}

\begin{cor}\label{help1}
    If $M \prec_{\mathbf{K}_{\alpha}} N$, then
    \begin{enumerate}
        \item If $|L_{\beta}^{M}| = \aleph_{\beta}$ for every $\beta \leq \alpha$ and $L_{\alpha + 1}^{M} = L_{\alpha + 1}^{N}$, then $N$ differs from $M$ only in the maximal branches it contains.
        \item If $|L_{\beta}^{M}| = \aleph_{\beta}$ for every $\beta \leq \alpha + 1$ and $L_{\alpha + 1}^{N}$ is a strict end extension of $L_{\alpha + 1}^{M}$, then $L_{\alpha + 1}^{M}$ does not have a maximum element and $L_{\alpha + 1}^{N}$ is a one-point end-extension of $L_{\alpha + 1}^{M}$.
        \item If $|L_{\beta}^{M}| = \aleph_{\beta}$ for every $\beta \leq \alpha$, $L_{\alpha + 1}^{M}$ has a maximum element and $L_{\alpha + 1}^{N}$ is a strict end extension of $L_{\alpha + 1}^{M}$, then $|M| = \aleph_{\alpha}$.
    \end{enumerate}
\end{cor}

\begin{proof}
    (1) It is immediate from Lemma \ref{help}.

    (2) First, observe that if $L_{\alpha + 1}^{M}$ had a maximum element $m$, then $|(L_{\alpha + 1}^{N})_{\leq_{(\alpha + 1)} m}| = \aleph_{\alpha + 1}$, which is a contradiction, since every initial segment of $L_{\alpha + 1}$ is bounded by $\aleph_{\alpha}$. We derive similar contradiction if we assume that there are at least two new elements in $L_{\alpha + 1}^{N}$.

    (3) First, we observe that $|L_{\alpha + 1}^{M}| \leq \aleph_{\alpha}$, from (2). Then we know from definition that $|V^{M}(x)| \leq \aleph_{\alpha}$, for every non maximal $x \in L_{\alpha + 1}^{M}$ and for the maximal level $V^{M}(m)$ the result comes from the fact that in the extension, $N$, $V^{N}(m)$ is not a maximal level any more, which implies that $|V^{N}(m)| \leq \aleph_{\alpha}$, since that is true for every non-maximal $V(x)$.

\end{proof}

\begin{lem}\label{maximal1}
    If $M$ is a maximal model of $\psi_\alpha$, then $|L_{\beta}^{M}| = \aleph_{\beta}$, for $\beta \leq \alpha$.
\end{lem}

\begin{proof}
    Towards contradiction take the least $\beta$, such that $|L_{\beta}^{M}| < \aleph_{\beta}$. If $\beta$ were a limit ordinal and since $L_{\beta}$ is the union of all previous $L_{\gamma}$'s, there should be a smaller ordinal satisfying $|L_{\gamma}^{M}| < \aleph_{\gamma}$. Therefore, $\beta$ is a successor ordinal and assume $\beta = \gamma + 1$. Then $|L_{\beta}^{M}| \leq \aleph_{\gamma}$, $|L_{\gamma}^{M}| = \aleph_{\gamma}$ and  we could extend $L_{\beta}^{M}$ by adding an element, $y$, greater than all the existing elements. Notice that, by Lemma \ref{help} (3), we cannot add any new element in the previous $L_{\delta}$'s, $\delta<\beta$, but we can add one in $L_{\beta}$, since the only thing we have to check is how we define the functions $F_{\beta}(x, \cdot, \cdot)$'s, for all $x \in L_{\beta + 1}$ and the function $F_{\gamma}(y, \cdot, \cdot)$ (since $y$ is greater than all the previous elements, we do not need to check the other $F_{\gamma}(x, \cdot, \cdot)$'s). 
    
    Defining the $F_{\beta}(x, \cdot, \cdot)$'s is easily done, since we only demand them to be surjections and we just added an element in the domain. For the function $F_{\gamma}(y, \cdot, \cdot)$, we just added one element in $L_{\beta}$, so we can find a surjection from $L_{\gamma}$ to $(L_{\beta})_{\leq_{\beta} y}$ Finally, for the $L_{\delta}$'s, for $\delta > \beta$, we do not add new elements, except when $\delta$ is a limit. So, every $L_{\delta}$ with $\delta > \beta$ limit,  contains the new element $y$. Then the linear order $<_{\delta}$ and the functions $F_{\delta}$ are easily defined.
\end{proof}

\begin{cor}\label{maximal2}
    Assume $\kappa \geq \aleph_{\alpha + 1}$, $M \in \mathbf{K}_{\alpha}$ is a tree that contains all its $\kappa$ cofinal branches and $|L_{\beta}^{M}| = \aleph_{\beta}$, for $\beta \leq \alpha$. Then $M$ is maximal.
\end{cor}

\begin{proof}
    Assume $M$ is not maximal. Then there exists $N \in \mathbf{K}_{\alpha}$ with $M \prec_{\mathbf{K}_{\alpha}} N$. Since $|L_{\beta}^{M}| = \aleph_{\beta}$ for every $\beta \leq \alpha$, $L_{\alpha + 1}^{M}$ has a maximum element and $|M| \geq \aleph_{\alpha + 1}$, then from Corollary \ref{help1}, $L_{\alpha + 1}^{N}$ cannot be a strict end extension of $L_{\alpha + 1}^{M}$. Therefore, $L_{\alpha + 1}^{M} = L_{\alpha + 1}^{N}$. By Corollary \ref{help1}, again, $M, N$ differ only on the maximal branches they contain. But $M$ contains all the maximal branches through the tree. Contradiction.
\end{proof}

\begin{lem}\label{lastdetails}
    Suppose $\alpha < \omega$. If $|L_{\beta}| = \aleph_{\beta}$, for some $\beta \leq \alpha + 1$, then $|L_{\gamma}| = \aleph_{\gamma}$, for every $\gamma \leq \beta$.
\end{lem}

\begin{proof}
    Suppose not. Then take $\gamma \leq \beta$ to be the least ordinal, such that $|L_{\gamma}| < \aleph_{\gamma}$. Then we can prove that $|L_{\delta}| < \aleph_{\delta}$ for every $\delta \geq \gamma$. If we prove that, we have come to a contradiction, since we know $|L_{\beta}| = \aleph_{\beta}$. We do that by induction on $\delta$ (note that it is a finite induction, since we just have finite ordinals). We have the result for $\gamma$. Now, suppose we have the result for some $\delta$ and we will prove it for $\delta + 1$. By the definition of the function $F_{\delta}$, we know that $|L_{\delta + 1}| \leq |L_{\delta}|^{+} < \aleph_{\delta}^{+} = \aleph_{\delta + 1}$.
\end{proof}

\underline{Observation:} We do not have the same result for $\alpha \geq \omega$, since we cannot apply the same induction argument for limit stages. For example we can have $\forall n > 0, |L_{n}| = \aleph_{n - 1}$, but we will have $|L_{\omega}| = \aleph_{\omega}$, since it is the union of all the previous $L_{n}$'s. 

\begin{lem}\label{height}
    If $M \in \mathbf{K}_{\alpha}$ and $|M| = \kappa > 2^{\aleph_{\alpha}}$, then $|L_{\alpha + 1}| = \aleph_{\alpha + 1}$.
\end{lem}

\begin{proof}
First, we observe that for every $\beta \leq \alpha, |L_{\beta}| \leq \aleph_{\alpha}$, and suppose that $|L_{\alpha + 1}| \leq \aleph_{\alpha}$. Then our model codes a tree with height at most $\aleph_{\alpha}$ and every non maximal level has at most $\aleph_{\alpha}$-many elements. That means that the whole model without the maximal level cannot exceed $\aleph_{\alpha}$. It also means that the maximal level has at most $\aleph_{\alpha}^{\aleph_{\alpha}} = 2^{\aleph_{\alpha}}$-many elements. Contradiction.
\end{proof}

In the previous Lemma, we can see that if $\kappa \leq 2^{\aleph_{\alpha}}$, then it is consistent that we can find a model of size $\kappa$ with $|L_{\alpha + 1}| \leq \aleph_{\alpha}$. For example, if we assume that $2^{< \aleph_{\alpha}} = \aleph_{\alpha}$ we can take the tree $(2^{\leq \aleph_{\alpha}}, \subset)$, with only $\kappa$ of its maximal branches.

\section{An Abstract Elementary Class}

\begin{thm}\label{AEC}
($\mathbf{K}_{\alpha}, \prec_{\mathbf{K}_{\alpha}}$) is an Abstract Elementary Class (AEC) with countable Lowenheim-Skolem number.
\end{thm}

\begin{proof}
    Most properties of being an AEC are easy to prove. We only prove that $LS(\mathbf{K}_{\alpha}) = \aleph_{0}$. 
    
    Consider $N \in \mathbf{K}_{\alpha}$ and some $X \subseteq N$. 
    We define by induction on $n \in \omega$ sets $L_{\beta}^{n} \subseteq L_{\beta}^{N}$, for every $\beta \leq \alpha + 1$, and  sets $V^{n}(x) \subseteq V^{N}(x)$ and take their unions to create the desired model, $M$.

    Define $L_{1}^{0}$ to be an initial segment of $L_{1}^{N}$ that contains $L_{1}^{N} \cap X$. We observe that $|L_{1}^{0}| \leq |X| + \aleph_{0}$, since every strict initial segment of $L_{1}$ is countable ($L_{1}$ is $\aleph_{1}$-like).

    For every $2 \leq \beta \leq \alpha$, we define $L_{\beta}^{0} = L_{\beta}^{N} \cap X$.

    Furthermore, we take $L_{\alpha + 1}^{0}$ to satisfy the following:

    \begin{itemize}
        \item $(L_{\alpha + 1}^{N} \cap X) \setminus \{m\} \subseteq L_{\alpha + 1}^{0}$, where $m$ is the maximum element of $L_{\alpha + 1}^{N}$ (if exists).
        \item $0 \in L_{\alpha + 1}^{0}$
        \item not to contain the maximum element (if any)        
        \item if $y \in V^{N} \cap X$ non maximal, then the level $x$, such that $y \in V^{N}(x)$, belongs to $L_{\alpha + 1}^{0}$
        \item if $y \neq z \in V^{N} \cap X$ belong to some limit level of $L_{\alpha + 1}^{N}$ (even if it is the maximum level), then they must differ at some level of $L_{\alpha + 1}^{0}$ (If $L_{\alpha + 1}^{0}$ does not already contain such a level, we introduce one to it.)
        \item $L_{\alpha + 1}^{0}$ is  closed for successors and predecessors. I.e. if $x \in L_{\alpha + 1}^{0}$, then $S(x) \in L_{\alpha + 1}^{0}$ and if $x = S(y)$, then $y \in L_{\alpha + 1}^{0}$
    \end{itemize}

    We observe that $|L_{\alpha + 1}^{0}| \leq |X| + \aleph_{0}$.

    Finally, for every $x \in L_{\alpha + 1}^{N} \cap X$, we define $V^{0}(x)$ as follows:

    \begin{itemize}
        \item $V^{0}(x)$ must contain $V^{N}(x) \cap X$ (could be empty)
        \item For every $x \in (L_{\alpha + 1}^{N} \cap X) \cup \{m\}$ (if there is a maximum), if $y \in V^{N}(x) \cap X$, then for every $x' <_{\alpha + 1} x$, with $x' \in L_{\alpha + 1}^{N} \cap X$, we choose some $z \in V^{N}(x')$ with $T(z, y)$ and require $z \in V^{0}(x')$. Similarly for every $x' >_{\alpha + 1} x$, $x' \in (L_{\alpha + 1}^{N} \cap X) \setminus \{m\}$ we choose some $z \in V^{0}(x')$ with $T(y, z)$. For the lower levels the choice is unique, but for the upper levels the choice may not be unique. We make the choice such that for every $y$ to create a path of the tree. This path does not contain any elements that belong to the maximal level. 
        
    \end{itemize}

    We also observe that $|V^{0}(x)| \leq |X| + \aleph_{0}$.

    Now, suppose we have defined $L_{\beta}^{n}$, for every $1 \leq \beta \leq \alpha + 1$, and $V^{n}(x)$, for every $x \in L_{\alpha + 1}^{n - 1}$, so that $|L_{\beta}^{n}|, |V^{n}(x)| \leq |X| + \aleph_{0}$.

    We define $L_{\beta}^{n + 1}$, for every $1 \leq \beta \leq \alpha$, by induction on $\beta$. 

    For our convenience, in this proof, when we refer to the functions $F_{\beta, x}$ and $G_{x}$, we mean the functions $F_{\beta, x} = F_{\beta}^{N}(x, \cdot, \cdot) : L_{\beta}^{N} \to (L_{\beta + 1}^{N})_{\leq_{(\beta + 1)} x}$ and $G_{x} = G^{N}(x, \cdot, \cdot) : L_{\alpha}^{N} \to V^{N}(x)$.

    In order to define $L_{1}^{n + 1}$, for every $x \in L_{2}^{n}$ and every $z \leq_{2} x, z \in L_{2}^{n}$, we choose a $y \in F_{1, x}^{-1}(\{z\})$. 
   
    So $L_{1}^{n + 1}$ is the least initial segment that contains $L_{1}^{n}$ and all the inverse images we have chosen. We observe that $|L_{1}^{n + 1}| \leq |X| + \aleph_{0}$, since the inverse images we chose are at most $|L_{2}^{n}| \cdot |L_{2}^{n}| \leq |X| + \aleph_{0}$.

    Suppose we have already defined $L_{\beta}^{n + 1}$ and we want to define $L_{\beta + 1}^{n + 1}$. Define $L_{\beta + 1}^{n + 1}$ to be the subset of $L_{\beta + 1}^{N}$ that contains $L_{\beta + 1}^{n}$ and which is closed under the following two operations: First, for every $x \in L_{\beta + 2}^{n}$ and every $z \leq_{\beta + 2} x, z \in L_{\beta + 2}^{n}$, we choose a $y \in F_{\beta + 1, x}^{-1}(\{z\})$. Second, we require that the images of all the functions $F_{\beta, x}|_{L_{\beta}^{n + 1}}$, for every $x \in L_{\beta + 1}^{n}$ are in $L_{\beta+1}^{n+1}$. We observe that $|L_{\beta + 1}^{n + 1}| \leq |X| + \aleph_{0}$.

    Finally, for limit stages, i.e. $\beta$ is limit and we have defined $L_{\gamma}^{n + 1}, \gamma < \beta$, for every $x \in L_{\beta + 1}^{n}$ and every $z \leq_{\beta + 1} x, z \in L_{\beta + 1}^{n}$, we choose a $y \in F_{\beta, x}^{-1}(\{z\})$. 
    So let $L_{\beta}^{n + 1}$ be the union of $L_{\beta}^{n}$, all $L_{\gamma}^{n + 1}, \gamma < \beta$ and all the inverse images we have chosen. 

    Observe that $|L_{\beta}^{n + 1}| \leq |X| + \aleph_{0}$.

    Now it remains to define $L_{\alpha + 1}^{n + 1}$ and $V^{n + 1}(x)$. We define $L_{\alpha + 1}^{n + 1}$ as follows:

    \begin{itemize}
        \item contains $L_{\alpha + 1}^{n}$
        \item contains all the images of the functions $F_{\alpha, x}|_{L_{\alpha}^{n + 1}}$, for every $x \in L_{\alpha + 1}^{n}$
        \item if $y \neq z \in V^{n}(x)$, where $x \in L_{\alpha + 1}^{n - 1}$ is %(in order $V^{n}(x)$ to be defined) 
        some limit point of $L_{\alpha + 1}^{N}$, then $y, z$ must differ at some level of $L_{\alpha + 1}^{n + 1}$
        \item it must be closed for successors and predecessors 
    \end{itemize}

    We have that $|L_{\alpha + 1}^{n + 1}| \leq |X| + \aleph_{0}$.

    Now, for every $x \in L_{\alpha + 1}^{n}$ we define $V^{n + 1}(x)$:
    
    \begin{itemize}
        \item contains $V^{n}(x)$ (if exists)
        \item contains all the images of $G_{x}|_{L_{\alpha + 1}^{n + 1}}$
        \item for every level, $y$ of $L_{\alpha + 1}^{n}$ or the maximum level, $m$, and for every element of those levels that belongs to $V^{n}(y)$ or $V^{N}(m) \cap X$, we choose a path of the tree containing that element and we require the new elements to be in the corresponding levels, $V^{n + 1}(\cdot)$. 
    \end{itemize}

    We notice that $|V^{n + 1}(x)| \leq |X| + \aleph_{0}$.

    Now, we are ready to define:
    \begin{itemize}
        \item $L_{\beta}^{\omega} = \bigcup_{n \in \omega} L_{\beta}^{n}$, for $1 \leq \beta \leq \alpha + 1$
        \item for every $x \in L_{\alpha + 1}^{\omega}, V^{\omega}(x) = \bigcup_{n \in \omega} V^{n}(x)$ (if for some $n, V^{n}(x)$ is not defined, we consider it as the empty set)  
    \end{itemize}

    We can see that $V^{\omega}(x)$ is well defined. More specifically, we need to check that for every $x \in L_{\alpha + 1}^{\omega}$, we have defined $V^{n}(x)$ for some $n$. If $x \in L_{\alpha + 1}^{\omega}$, then there is some $n$ such that $x \in L_{\alpha + 1}^{n}$ and therefore we have defined $V^{n + 1}(x)$.

    \underline{Some observations for $L_{\alpha + 1}^{\omega}$ :}
    \begin{enumerate}
        \item $0 \in L_{\alpha + 1}^{\omega}$ and is its minimum element
        \item the restriction of $<_{\alpha + 1}$ on $L_{\alpha + 1}^{\omega}$ is a linear order
        \item $L_{\alpha + 1}^{\omega}$ is closed under successor function 
        \item $m \notin L_{\alpha + 1}^{\omega}$ (if exists)
        \item $|L_{\alpha + 1}^{\omega}| \leq |X| + \aleph_{0}$ (that is also holds for every $\beta \leq \alpha$)
        \item if $x \in L_{\alpha + 1}^{\omega}$, then $x$ is limit for $L_{\alpha + 1}^{\omega}$ iff $x$ is limit for $L_{\alpha + 1}^{N}$ (that comes from the fact that every $L_{\alpha + 1}^{n}$ is closed for successors and predecessors)
    \end{enumerate}

    We are, finally, ready to define the desired model $M$ :

    \begin{itemize}
        \item $L_{0}^{M} = L_{0}^{N}$
        \item $L_{\beta}^{M} = L_{\beta}^{\omega}$, for every $1 \leq \beta \leq \alpha$
        \item $L_{\alpha + 1}^{M} = L_{\alpha + 1}^{\omega} \cup \{m\}$
        \item $V^{M} = \bigcup_{x \in L_{\alpha + 1}^{\omega}} V^{\omega}(x) \cup (V^{N}(m) \cap X)$
    \end{itemize}

    The universe of $M$ is the union of these sets.

    It remains to prove that $M \prec_{\mathbf{K}_{\alpha}} N$:

    We prove some closure properties and the rest of the details are left for the reader.
    
    First, we have to check if the restrictions of the functions have the properties we want. 

    For every $x \in L_{1}^{\omega}, F_{0, x}^{M} = F_{0, x}^{N}$, which is well defined, since we took initial segments for each $L_{1}^{n}$ (so $L_{1}^{\omega}$ is an initial segment of $L_{1}^{N}$).

    For $1 \leq \beta \leq \alpha$, we take as $F_{\beta, x}^{M}$ the restriction of $F_{\beta, x}^{N}$ on $L_{\beta}^{\omega}$, for every $x \in L_{\beta + 1}^{\omega}$. We need to check it is well defined, i.e. if $y \in L_{\beta}^{\omega}$, then $F_{\beta, x}^{M}(y)$ should be in $L_{\beta + 1}^{M} = L_{\beta + 1}^{\omega}$. Indeed, since $x \in L_{\beta + 1}^{\omega}$ and $y \in L_{\beta}^{\omega}$, then $x \in L_{\beta + 1}^{n}$ and $y \in L_{\beta}^{n + 1}$, for some $n$ and that means, by definition, that its image through $F_{\beta}$ is in $L_{\beta + 1}^{n + 1} \subseteq L_{\beta + 1}^{\omega}$. We, also, need to check that the function is onto. If $y \in (L_{\beta + 1}^{\omega})_{\leq_{(\beta + 1)} x}$, we will find some $z \in L_{\beta}^{\omega}$, such that $F_{\beta, x}^{M}(z) = y$. Since $x \in L_{\beta + 1}^{\omega}$ and $y \in L_{\beta}^{\omega}$, there is a $n$, with $x \in L_{\beta + 1}^{n}$ and $y \in L_{\beta}^{n + 1}$. By definition $L_{\beta}^{n + 1} \subseteq L_{\beta}^{\omega}$ contains some $z \in F_{\beta}^{-1}(\{y\})$.

    Similarly, we can prove the properties we want for $G_{x}$, for every $x \in L_{\alpha + 1}^{\omega}$.

    All the other requirements ($m$ is not a successor, elements at limit levels must differ at some lower level, tree ordering, ...) for our model are immediate from how we defined the universe of $M$.
    
\end{proof}

\section{Spectrum results}

The next theorem characterizes the spectrum of $\psi_{\alpha}$.

\begin{thm}\label{spectrum}
    The spectrum of $\psi_{\alpha}$ is characterized by the following properties:
    \begin{enumerate}
        \item $[\aleph_{0}, \aleph_{\alpha}^{\aleph_{0}}] \subseteq Spec(\psi_{\alpha})$ and $\aleph_{\alpha + 1} \in Spec(\psi_{\alpha})$.
        \item if there exists a $(\mu, \lambda, \kappa)$-Kurepa tree, where $\aleph_{1} \leq \mu \leq \lambda \leq \aleph_{\alpha + 1}$, then $[\aleph_{0}, \kappa] \subseteq Spec(\psi_{\alpha})$. (Note here that this case includes $\aleph_{\alpha + 1}$-Kurepa trees, when $\mu$ equals $\aleph_{\alpha + 1}$.)
        \item no cardinal belongs to $Spec(\psi_{\alpha})$ except those required by (1)-(2). I.e. if $\psi_{\alpha}$ has a model of size $\kappa$, then $\kappa \in [\aleph_{0}, max\{\aleph_{\alpha}^{\aleph_{0}}, \aleph_{\alpha + 1}\}]$ or there exists a $(\mu, \lambda, \kappa)$-Kurepa tree, $\mu \leq \lambda\leq \aleph_{\alpha + 1}$.
    \end{enumerate}
\end{thm}

\begin{proof}
    (1) It is easy to see that $(\aleph_{\alpha}^{\leq \omega}, \subset)$ is a model of $\psi_{\alpha}$ with size $\aleph_{\alpha}^{\aleph_{0}}$. In fact, that's a tree with countable height and every non maximal level has size of $\aleph_{\alpha}$. 
    
    It remains to construct a model of $\psi_{\alpha}$ with size $\aleph_{\alpha + 1}$. For example a tree with height $\aleph_{\alpha + 1}$ and each level has only one element is such a model. 
    
    (2) It is immediate from the definition that every $(\mu, \lambda, \kappa)$-Kurepa tree with $\mu \leq \lambda \leq \aleph_{\alpha + 1}$ gives rise to a model of $\psi_{\alpha}$.

    (3) Assume $N \models \psi_{\alpha}$ and $|N| = \kappa$.

    \begin{itemize}
        \item If $|L_{\alpha + 1}^{N}| \leq \aleph_{0}$, then $\kappa \leq \aleph_{\alpha}^{\aleph_{0}}$. 

        This comes from the fact that $\forall \beta \leq \alpha, |L_{\beta}| \leq \aleph_{\beta} \leq \aleph_{\alpha}$ and that the tree has countable size with levels of size at most $\aleph_{\alpha}$.
        \item If $\aleph_{1} \leq |L_{\alpha + 1}^{N}| \leq \aleph_{\alpha}$, then we need to check its cofinality. If the cofinality of $L_{\alpha + 1}^{N}$ is $\aleph_{0}$, then $\kappa \leq \aleph_{\alpha}^{\aleph_{0}}$ and this falls under the first case. If the cofinality is $\aleph_{1} \leq \mu \leq \aleph_{\alpha}$, then we have two subcases: 
        \begin{itemize}
            \item if the tree has $\leq \aleph_{\alpha}$ cofinal branches, then $\kappa \leq \aleph_{\alpha}$ and so we are again in the case (1)
            \item if the tree has at least $\aleph_{\alpha + 1}$ cofinal branches, then $N$ codes a $(\mu, \aleph_{\alpha + 1}, \kappa)$-Kurepa tree.
        \end{itemize}
        \item If $|L_{\alpha + 1}^{N}| = \aleph_{\alpha + 1}$, we have two subcases:
        \begin{itemize}
            \item if $\kappa = \aleph_{\alpha + 1}$, we are in case (1)
            \item If $\kappa > \aleph_{\alpha + 1}$, then $N$ codes an $\aleph_{\alpha + 1}$-Kurepa tree. 
        \end{itemize}
    \end{itemize}
\end{proof}

Next we prove a theorem for the maximal models spectrum of $\psi_{\alpha}$.

\begin{thm}\label{maximal model spectrum}
The maximal models Spectrum of $\psi_{\alpha}$ is characterized by the following:
\begin{enumerate}
    \item $\psi_{\alpha}$ has a maximal model of size $\aleph_{\alpha + 1}$
    \item If $\lambda^{\aleph_{0}} \geq \aleph_{\alpha + 1}$, for some $\aleph_{0} \leq \lambda \leq \aleph_{\alpha}$, then $\psi_{\alpha}$ has a maximal model of size $\lambda^{\aleph_{0}}$
    \item If there exists a $(\mu, \aleph_{\alpha + 1}, \kappa)$-Kurepa tree, $\aleph_{1} \leq \mu \leq \aleph_{\alpha + 1}$, with exactly $\kappa$ cofinal branches, then $\psi_{\alpha}$ has a maximal model in $\kappa$ 
    \item $\psi_{\alpha}$ has maximal models only on those cardinalities required by (1)-(3).
\end{enumerate}
\end{thm}

\begin{proof}
    First observe that if a model is maximal, then by Lemma \ref{maximal1} $|L_{\beta}| = \aleph_{\beta}$, for every $\beta \leq \alpha$. So, in this proof we will require every model to satisfy this requirement. 

    (1) We can easily construct a tree of height $\aleph_{\alpha + 1}$ with $\aleph_{\alpha + 1}$ cofinal branches. Then by Corollary \ref{maximal2} we have that this is a maximal model of $\psi_{\alpha}$.
    
    (2) We can easily see that $(\lambda^{\leq \omega}, \subset)$ is a model of $\psi_{\alpha}$. Now, if in addition, we know that $\lambda^{\aleph_{0}} \geq \aleph_{\alpha + 1}$, then again from Corollary \ref{maximal2}, we see that it is a maximal model.

    (3) We have already seen that those trees are models of our sentence and the maximality follows from Corollary \ref{maximal2}.

    (4) First, observe that we always can extend a model of cardinality $\leq \aleph_{\alpha}$.

    Assume, now that $N$ is a maximal model of $\psi_{\alpha}$.

    \begin{itemize}
       \item If $|L_{\alpha + 1}^{N}| = \aleph_{0}$, then $N$ codes a tree with countable height and levels with at most $\aleph_{\alpha}$ elements. Then we can prove that its maximal branches are either $\lambda$ or $\lambda^{\aleph_{0}}$, where $\lambda$ is the size of the body of the tree. Indeed, from \cite{Handbookofsettheoretictopology}, Theorem 8.3, we know that if $X$ is a completely metrizable space with weight $\lambda$ (a space with weight $\lambda$ is a space that the smallest cardinal for an open base for its topology is $\lambda$, which is equivalent to say that the smallest dense subset of the space is of size $\lambda$), then $|X| = \lambda$ or $|X| = \lambda^{\aleph_{0}}$. 
     
       If $X$ is the set of the maximal branches of the tree, then we can define a complete metric on it as follows: if $x, y \in X$, $d(x, y) = 1/n$, where $n$ is the least level that these elements differ. We can also see that the smallest cardinality of a dense subset is $\lambda$, for example, for every element of the tree, $b$, chose a branch $x_{b} \in X$ that contains $b$ and the desired subset is $\{ x_{b} | b \in T \}$ ($T$ is the body of the tree).

       So this proves that the tree (with its cofinal branches) is of size $\lambda$ or $\lambda^{\aleph_{0}}$ and the whole model is of size $\aleph_{\alpha}$ or $\lambda^{\aleph_{0}}$ (if $\lambda^{\aleph_{0}} \geq \aleph_{\alpha + 1}$). 

       Finally, from Corollary \ref{maximal2}, if the height of our tree is countable we can only have maximal model of size $\lambda^{\aleph_{0}}$ and so we are in case (2).

       \item If $\aleph_{1} \leq |L_{\alpha + 1}^{N}| \leq \aleph_{\alpha}$, then we have to check its cofinality. If the cofinality is countable, then we work with a cofinal subset of $L_{\alpha + 1}^{N}$ as in the previous case. If the cofinality is some uncountable $\mu$, we have two more subcases. If the maximal branches of the tree are $\leq \aleph_{\alpha}$, then $|N| = \aleph_{\alpha}$ (similarly to corresponding case of the proof of Theorem \ref{spectrum}) which means the model is not maximal. If the maximal branches of the tree are at least $\aleph_{\alpha + 1}$, lets say $\kappa$, then our model codes a $(\mu, \aleph_{\alpha + 1}, \kappa)$-Kurepa tree. 
       \item If $|L_{\alpha + 1}^{N}| = \aleph_{\alpha + 1}$, then if $|N| = \aleph_{\alpha + 1}$, we are done. If $|N| \geq \aleph_{\alpha + 2}$, then the maximal branches of the tree should be $\geq \aleph_{\alpha + 2}$ and that means our model codes an $\aleph_{\alpha + 1}$-Kurepa tree.
    \end{itemize}
\end{proof}
\begin{cor}\label{results}
    \begin{enumerate}
        \item If there are no $(\mu, \lambda, \kappa)$-Kurepa trees, with $\mu, \lambda \leq \aleph_{\alpha + 1}$ (this case includes $\aleph_{\alpha + 1}$-Kurepa trees), then $Spec(\psi_{\alpha}) = [\aleph_{0}, max\{\aleph_{\alpha}^{\aleph_{0}}, \aleph_{\alpha + 1}\}]$ and $MM-Spec(\psi_{\alpha}) = \{ \lambda^{\aleph_{0}} | \aleph_{0} \leq \lambda \leq \aleph_{\alpha}$ and $\lambda^{\aleph_{0}} \geq \aleph_{\alpha + 1}\} \cup \{\aleph_{\alpha + 1}\}$.
        \item If $\mathcal{B}(\aleph_{\alpha + 1}) > 2^{\aleph_{\alpha}}$ is a maximum, i.e. there is an $\aleph_{\alpha + 1}$-Kurepa tree of size $\mathcal{B}(\aleph_{\alpha + 1})$, then $\psi_{\alpha}$ characterizes $\mathcal{B}(\aleph_{\alpha + 1})$.
        \item If $\mathcal{B}(\aleph_{\alpha + 1}) > 2^{\aleph_{\alpha}}$ is not a maximum, 
        then $Spec(\psi_{\alpha})$ equals $[\aleph_{0}, \mathcal{B}(\aleph_{\alpha + 1}) )$. Moreover, $\psi_{\alpha}$ has maximal models in $\aleph_{\alpha + 1}, in \lambda^{\aleph_{0}}$, where $\lambda^{\aleph_{0}}$ is as in case (1),  and in cofinally many cardinalities below $\mathcal{B}(\aleph_{\alpha + 1})$.
        
    \end{enumerate}
\end{cor}

\begin{proof}
    (1) and (2) follow immediately from Theorems \ref{spectrum} and \ref{maximal model spectrum}. We only establish (3). The result for the spectrum is also immediate from Theorem \ref{spectrum}. For the last assertion, assume $\mathcal{B}(\aleph_{\alpha + 1}) = sup_{i} \kappa_{i}$ and for each $i$, there is an $\aleph_{\alpha + 1}$-Kurepa tree with exactly $\kappa_{i}$ maximal branches. Then, from Theorem \ref{maximal model spectrum} (3), each $\kappa_{i}$ is in the $MM-Spec(\psi_{\alpha})$.

    Observe that the assumption $\mathcal{B}(\aleph_{\alpha + 1}) > 2^{\aleph_{\alpha}}$ in cases (2) and (3) is used to ensure that there is no model of our sentence  which codes a tree with height $\mu \leq \aleph_{\alpha}$ and whose size exceeds $\mathcal{B}(\aleph_{\alpha + 1})$.
\end{proof}

Now we provide the amalgamation and joint embedding spectrum of models of $\psi_{\alpha}$.

\begin{thm}\label{JEPAP}
    \begin{enumerate}
        \item $(\mathbf{K}_{\alpha}, \prec_{\mathbf{K}_{\alpha}})$ fails JEP in all cardinals.
        \item \begin{itemize}
            \item If $0 < \alpha < \omega$ and $\kappa$ some infinite cardinal that belongs to $Spec(\psi_{\alpha})$, then
            \begin{itemize}
                \item if $\kappa \leq \aleph_{\alpha}$, $(\mathbf{K}_{\alpha}, \prec_{\mathbf{K}_{\alpha}})$ does not satisfy AP
                \item if $\aleph_{\alpha + 1} \leq \kappa \leq 2^{\aleph_{\alpha}}$, it depends whether $(\mu, \aleph_{\alpha}, \kappa)$-Kurepa trees, for $\mu \leq \aleph_{\alpha}$ exist or not. More specifically if there exists such tree for some $\mu$, then AP fails for $\kappa$; otherwise AP holds.
                \item if $\kappa > 2^{\aleph_{\alpha}}$, then $(\mathbf{K}_{\alpha}, \prec_{\mathbf{K}_{\alpha}})$ satisfies AP.
            \end{itemize}
            \item If $\alpha \geq \omega$, then $(\mathbf{K}_{\alpha}, \prec_{\mathbf{K}_{\alpha}})$ fails AP in all cardinalities.
        \end{itemize}
    \end{enumerate}
\end{thm}

\begin{proof}

    (1)     The first observation is that in all cardinalities there exist two linear orders none of which is initial segment of the other.

Assume $M, N \models \psi_{\alpha}$, but neither $L_{1}^{N}$ is not an initial segment of $L_{1}^{M}$ nor $L_{1}^{M}$ is not an initial segment of $L_{1}^{N}$. Then by Lemma \ref{help} (2), $M, N$ cannot be jointly embedded to some larger structure in $\mathbf{K}_{\alpha}$.

    (2) Assume $M_{0} \prec_{\mathbf{K}_{\alpha}} M_{1}, M_{2} \in \mathbf{K}_{\alpha}$ models of size $\kappa$. 

    Although we will treat each case separately, the idea is the same. We know that in all cardinalities there exists two linear orders none of which is an initial segment of the other. So, if we can find some $L_{\beta}^{M_{0}}$ that can be extended properly, we use that argument to find two extensions that cannot be embedded in a larger one. Otherwise, we can find the amalgam of $(M_{0}, M_{1}, M_{2})$. So, the same argument or counterexample will be used to prove many of the following cases.
    
    \underline{Case 1: $0 < \alpha < \omega$.}

    \begin{itemize}
        \item Suppose $\kappa \leq \aleph_{\alpha}$, lets say that $\kappa = \aleph_{\gamma}$, for some $\gamma \leq \alpha$. Then take $M_{0}$ to be the model that consists of $L_{\beta} = \omega_{\beta},$ for $0 \leq \beta \leq \gamma$ and $L_{\beta} = \omega_{\gamma}$, for $\gamma + 1 \leq \beta \leq \alpha + 1$ and every level $V^{M_0}(x)$ contains only one element. 
     
        Then we can define $M_{1}, M_{2}$ as follows. For $\beta \leq \gamma + 1$, $L_{\beta}^{M_{1}} = L_{\beta}^{M_{2}} = L_{\beta}^{M_{0}}$ and for $\beta \geq \gamma + 1$, $L_{\beta}^{M_{1}} = \omega_{\gamma} + \omega$, $L_{\beta}^{M_{2}} = \omega_{\gamma} + \mathbb{Q}$ and every level (both the old ones and the new ones) again contains one element. 
      
        It is easy to see that these are models of $\psi_{\alpha}$. (I mentioned nothing about the linear orders and the functions $F_{\beta}$, but I assume the order to be the usual one on every $L_{\beta}$ and it is easy to define the desired functions. In fact you can define a lot of different functions satisfying  the requirements of $\psi_{\alpha}$.)
        
        By Lemma \ref{help} (4), we know that since in $M_{0}$ we have $|L_{\beta}| = \aleph_{\beta}$, for all $\beta \leq \gamma$, then $L_{\gamma + 1}^{M_{0}}$ should be an initial segment of $L_{\gamma + 1}$ in every extension. So, similarly as in case (1) $L_{\gamma + 1}^{M_{0}}$ is an initial segment of both $L_{\gamma + 1}^{M_{1}}$ and $L_{\gamma + 1}^{M_{2}}$, but neither $L_{\gamma + 1}^{M_{1}}$ is an initial segment of $L_{\gamma + 1}^{M_{2}}$ nor $L_{\gamma + 1}^{M_{2}}$ is an initial segment of $L_{\gamma + 1}^{M_{1}}$. Thus, the triple $(L_{\gamma + 1}^{M_{0}}, L_{\gamma + 1}^{M_{1}}, L_{\gamma + 1}^{M_{2}})$ cannot be amalgamated.

    \item Suppose $\aleph_{\alpha + 1} \leq \kappa \leq 2^{\aleph_{\alpha}}$. Our model could reach size $\kappa$ only if $|L_{\alpha + 1}| = \aleph_{\alpha + 1}=\kappa$, or if our tree has $\kappa$ cofinal branches. 

    If there exists some $\mu \leq \aleph_{\alpha}$ with a $(\mu, \aleph_{\alpha}, \kappa)$-Kurepa tree, then we take as $M_{0}$ the following model: Define $L_{0} = \omega, L_{\beta} = \omega_{\beta}$, for $1 \leq \beta \leq \alpha - 1$, $L_{\alpha} = \omega_{\alpha - 1}$, $L_{\alpha + 1} = \mu \cup \{m\}$ (we include into our model the maximal level of the tree) and the levels of the model are the levels of the $(\mu, \aleph_{\alpha}, \kappa)$-Kurepa tree including the maximal one. It is easy to see that this is a model of $\psi_{\alpha}$. 
    
    Then, define as $M_{1}$ and $M_{2}$ the following. For $0 \leq \beta \leq \alpha - 1$, $L_{\beta}^{M_{1}} = L_{\beta}^{M_{2}} = L_{\beta}^{M_{0}}$, $L_{\alpha}^{M_{1}} = \omega_{\alpha - 1} + \omega$, $L_{\alpha}^{M_{2}} = \omega_{\alpha - 1} + \mathbb{Q}$ and $L_{\alpha + 1}^{M_{1}} = L_{\alpha + 1}^{M_{2}} = L_{\alpha + 1}^{M_{0}}$. The levels of $M_{1}, M_{2}$ are the same levels that $M_{0}$ has.
   With a similar argument to the previous case, we see that the triple $(L_{\alpha}^{M_{0}}, L_{\alpha}^{M_{1}}, L_{\alpha}^{M_{2}})$ cannot be amalgamated.  
    
    If there are no such trees, but we have a model in $\kappa$, we take two more subcases:
    \begin{itemize}
        \item If $|L_{\alpha + 1}^{M_{0}}| = \aleph_{\alpha + 1}$, then from Lemma \ref{lastdetails}, we have that $|L_{\beta}^{M_{0}}| = \aleph_{\beta}$ for every $\beta \leq \alpha + 1$. So by Lemma \ref{help}, we have that $L_{\beta}^{M_{0}} = L_{\beta}^{M_{1}} = L_{\beta}^{M_{2}}$, for every $\beta \leq \alpha$ and from Corollary \ref{help1} (3) (since $|M_{0}| > \aleph_{\alpha}$), that either $L_{\alpha + 1}^{M_{1}}, L_{\alpha + 1}^{M_{2}}$ are equal to $L_{\alpha + 1}^{M_{0}}$ or they are one point extensions. 
        
        \item If $|L_{\alpha + 1}^{M_{0}}| \leq \aleph_{\alpha}$, then the tree that $M_{0}$ codes has $\kappa \geq \aleph_{\alpha + 1}$ maximal branches. Since we are in case that there are no $(\mu, \aleph_{\alpha}, \kappa)$-Kurepa trees for $\mu \leq \aleph_{\alpha}$, then we can conclude that the tree that $M_{0}$ codes has level with $\aleph_{\alpha}$-many elements and therefore we can see that $|L_{\alpha}^{M_{0}}| = \aleph_{\alpha}$. So, by Lemma \ref{lastdetails}, we have that $|L_{\beta}^{M_{0}}| = \aleph_{\beta}$,  for $\beta \leq \alpha$. So by Lemma \ref{help}, we have that $L_{\beta}^{M_{0}} = L_{\beta}^{M_{1}} = L_{\beta}^{M_{2}}$, for every $\beta \leq \alpha$ and from Corollary \ref{help1} (3) (since $|M_{0}| > \aleph_{\alpha}$), that either $L_{\alpha + 1}^{M_{1}}, L_{\alpha + 1}^{M_{2}}$ are equal to $L_{\alpha + 1}^{M_{0}}$ or they are one point extensions. 
        
    \end{itemize}

    Now, in both of the above subcases, define the amalgam $N$ of $(M_{0}, M_{1}, M_{2})$ to be the union of $M_{0}$ together with all maximal branches in $M_{1}, M_{2}$ (if any). If two maximal branches have the same predecessors, we identify them. It follows that $N \in \mathbf{K}_{\alpha}$ and $M_{i} \prec_{\mathbf{K}_{\alpha}} N, i = 0, 1, 2$.
    
    \item Suppose $\kappa > 2^{\aleph_{\alpha}}$. Then, by Lemma \ref{height}, that means that our tree has height $\aleph_{\alpha + 1}$, or in other words $|L_{\alpha + 1}^{M_{i}}| = \aleph_{\alpha + 1}, i = 0, 1, 2$. So we work similarly to the previous case where we had $|L_{\alpha + 1}^{M_{0}}| = \aleph_{\alpha + 1}$ and find the amalgam of $(M_{0}, M_{1}, M_{2})$.
    \end{itemize}

    \underline{Case 2: $\alpha \geq \omega$.}

    Since we claim that AP fails for all cardinalities, we have to find a counterexample for each $\kappa$.

    Firstly, assume that $\kappa > \aleph_{\alpha}$ and $N$ is a model of $\psi_{\alpha}$ of size $\kappa$. We can define a new model, $M_{0}$ of size $\kappa$ as follows. We define $L_{0}^{M_{0}} = L_{1}^{M_{0}} = \omega$, $L_{n}^{M_{0}} = \omega_{n - 1}$ for every $1 < n < \omega$ (that gives us that $L_{\omega}^{M_{0}} = \omega_{\omega}$) and for each $\beta > \omega$ successor $L_{\beta}^{M_{0}} = L_{\beta}^{N}$ (the limit $L_{\beta}$'s are the union of all the previous, so they may not be the same they where for model $N$). Furthermore define $V(a)^{M_{0}} = V(a)^{N}$ for every $a \in L_{\alpha + 1}$ and observe that it is easy to define the functions $F_{\beta}$, for $\beta \leq \alpha$ and $G$. 
    
    Now, to complete the counterexample, we have to define two more models $M_{1}, M_{2}$. We take as $M_{1}, M_{2}$ the models that everything is the same as in the model $M_{0}$ except that $L_{1}^{M_{1}} = \omega + \omega$ and  $L_{1}^{M_{2}} = \omega + \mathbb{Q}$. We see that the triple $(L_{1}^{M_{0}}, L_{1}^{M_{1}}, L_{1}^{M_{2}})$ cannot be amalgamated. Therefore, our counterexample is the triplet $(M_{0}, M_{1}, M_{2})$.
    
    In the above example observe that the size of the models we define is still $\kappa$, since $\kappa > \aleph_{\alpha}$ is the size of $N$ and we just change the $L_{\beta}$'s for $\beta \leq \omega$.
    
    For $\kappa \leq \aleph_{\alpha}$ we can use the same counterexample as in the previous case.
\end{proof}

We can see that in the previous Theorem we excluded the case $\alpha = 0$. One can notice that if $\alpha = 0$, then the construction is the same as the construction in \cite{Kurepatrees}. All the results in \cite{Kurepatrees} apply for our paper in case $\alpha = 0$. 

More specifically, if $\alpha=0$, the amalgamation property fails for $\aleph_{0}$, while it holds for all uncountable cardinalities that belong to the spectrum.

\section{Consistency results}

In this section the Theorems \ref{forcing}, \ref{consistency}, are direct generalizations to higher cardinals, of Theorems 3.1(1) and 3.8 in \cite{Kurepatrees} . Their proofs are almost identical and included for completeness.

\begin{thm}\label{forcing}
    It is consistent with ZFC that $2^{\aleph_{\alpha}} < \aleph_{\omega_{\alpha + 1}} = \mathcal{B}(\aleph_{\alpha + 1}) < 2^{\aleph_{\alpha + 1}}$ and there exists an $\aleph_{\alpha + 1}$-Kurepa tree with $\aleph_{\omega_{\alpha + 1}}$-many cofinal branches.
\end{thm}

\begin{proof}
    Let $V_{0}$ be a model of ZFC+GCH and $P$ be the poset that contains conditions of the form $(t, f)$, where:

    \begin{itemize}
        \item $t$ is a tree of height $\beta + 1$ for some $\beta < \omega_{\alpha + 1}$ and levels of size at most $\aleph_{\alpha}$;
        \item $f$ is a function with $dom(f) \subset \aleph_{\omega_{\alpha + 1}}, |dom(f)| = \aleph_{\alpha}$ and $ran(f) = t_{\beta}$, where $t_{\beta}$ is the $\beta$th level of $t$.
    \end{itemize}

    We can think of $t$ is an initial segment of the generically added tree, and each $f(\delta)$ determines where the $\delta$th branch intersects the tree at level $\beta$. Now, we have to define the order: $(u, g) \leq (t, f)$ if,

    \begin{itemize}
        \item $t$ is an initial segment of $u, dom(f) \subset dom(g)$,
        \item for every $\delta \in dom(f)$, either $f(\delta) = g(\delta)$, if $t$ and $u$ have the same height, or $f(\delta) <_{u} g(\delta)$ otherwise.
        \end{itemize}

        We prove that $P$ is $\aleph_{\alpha + 1}$-closed and has the $\aleph_{\alpha + 2}$-chain condition. 
       
        For the $\aleph_{\alpha + 1}$-closed condition we prove that any decreasing sequence in $P$ of length $< \aleph_{\alpha + 1}$, has a lower bound in $P$. So assume we have such a sequence $(p_{\gamma} | \gamma < \tau)$, $\tau \leq \aleph_{\alpha}$ and for every $\gamma$, $p_{\gamma} = (t_{\gamma}, f_{\gamma})$. 
        
        In order to find the lower bound we define $t = \bigcup_{\gamma < \tau} t_{\gamma}$ (since each tree is initial segment of the next ones, we can do that), which is a tree of height $< \aleph_{\alpha + 1}$ with levels of size at most $\aleph_{\alpha}$. Note here that $t$ must contain a maximal element in order to be an element of $P$. If the union does not contain a maximal level, we add one. 
        
        Next we include maximal branches, but we have to be careful, since the maximal level of the tree maybe too big. We add just the elements we need. 
        More specifically, for every $\delta \in \bigcup_{\gamma} dom(f_{\gamma})$, we add a new element given by the branch $ \bigcup_{\gamma < \tau\text{ and } \delta \in dom(f_{\gamma})} f_{\gamma}(\delta)$.
             
        For the second coordinate, we define $f$ as follows: $dom(f) = \bigcup_{\gamma} dom(f_{\gamma})$ and for every $\delta \in dom(f)$, we take the branch $ \bigcup_{\gamma < \tau\text{ and } \delta \in dom(f_{\gamma})} f_{\gamma}(\delta)$ and define $f(\delta)$ to be the element of the $\beta$th level given by this branch. 
        
        In order to see the $\aleph_{\alpha + 2}$-chain condition, we check that every antichain has size at most $\aleph_{\alpha + 1}$. Towards contradiction suppose that $(p_{\gamma} | \gamma < \aleph_{\alpha + 2})$ is an antichain in $P$, where $p_{\gamma} = (t_{\gamma}, f_{\gamma})$. We see that there exist only $2^{\aleph_{\alpha}} = \aleph_{\alpha + 1}$-many possibilities for the first coordinate. So there are $\aleph_{\alpha + 2}$-many pairs with the same first coordinate and therefore, we assume that $t_\gamma$ is equal to some $t$ for all $\gamma$. 
        Since the $p_\gamma$'s are different, this means that the $f_\gamma$'s are different. 
        
        Take $W = \{ dom(f_{\gamma}) | \gamma<\aleph_{\alpha+2}\}$. We easily see that $|W| = \aleph_{\alpha + 2}$. From a $\Delta$-system lemma on $W$, we have that $\exists Z \subseteq W$, with $|Z| = \aleph_{\alpha + 2}$ such that $\forall \gamma_{1} \neq \gamma_{2}$, $dom(f_{\gamma_{1}}) \cap dom(f_{\gamma_{2}}) = R$. 

        Since $|R| \leq \aleph_{\alpha}$, the functions from $R$ to the $\beta^{th}$ level of $t$ are at most $\aleph_{\alpha}^{\aleph_{\alpha}} = 2^{\aleph_{\alpha}} = \aleph_{\alpha + 1}$. But we have $\aleph_{\alpha + 2}$-many functions in $Z$. Therefore, there exist $\gamma_{1} \neq \gamma_{2}$, such that $f_{\gamma_{1}}|_{R} = f_{\gamma_{2}}|_{R}$ and the pairs $(t_{\gamma_{1}}, f_{\gamma_{1}})$ and $(t_{\gamma_{2}}, f_{\gamma_{2}})$ are compatible. Contradiction. This finishes the proof that $P$ is $\aleph_{\alpha + 1}$-closed and has the $\aleph_{\alpha + 2}$-chain condition.

       Now, suppose that $H$ is $P$-generic over $V_{0}$. Then $\bigcup_{(t, f) \in H} t$ is an $\aleph_{\alpha + 1}$-Kurepa tree with $\aleph_{\omega_{\alpha + 1}}$-many branches, where for $\delta < \aleph_{\omega_{\alpha + 1}}$, the $\delta$th branch is given by $\bigcup_{(t, f) \in H, \delta \in dom(f)} f(\delta)$. These branches are distinct by standard density arguments. 

       We claim that in this model $\aleph_{\omega_{\alpha + 1}} = \mathcal{B}(\aleph_{\alpha + 1})$. That is because in the ground model all the $\aleph_{\alpha + 1}$-Kurepa trees are of size $\aleph_{\alpha + 2}$ and since the forcing notion is of size $\aleph_{\omega_{\alpha + 1}}$ and has the $\aleph_{\alpha + 2}$-chain condition cannot add an $\aleph_{\alpha + 1}$-Kurepa tree larger than $\aleph_{\omega_{\alpha + 1}}$.
       
       Also, note that since $\mathcal{B}(\aleph_{\alpha + 1})$ cannot exceed $2^{\aleph_{\alpha + 1}}$, we have that $\aleph_{\omega_{\alpha + 1}}= \mathcal{B}(\aleph_{\alpha + 1}) < 2^{\aleph_{\alpha + 1}}$ (the inequality is strict, because of cofinality considerations).
       %since we have $cf(\aleph_{\omega_{\alpha + 1}}) = \aleph_{\alpha + 1} < cf(2^{\aleph_{\alpha + 1}}))$. 
       
       Finally, note that the forcing notion is $\aleph_{\alpha + 1}$-closed and then it does not add new subsets of $\aleph_{\alpha}$, which means that $2^{\aleph_{\alpha}} = \aleph_{\alpha + 1}$ in the new model too. The model $V_{0}[H]$ proves the Theorem.
\end{proof}

\begin{cor}\label{last}
    In Theorem \ref{forcing}, we can replace $\aleph_{\omega_{\alpha + 1}}$ with any $\lambda \geq \aleph_{\alpha + 2}$ and $cf(\lambda) > \aleph_{\alpha}$. The only modification which might happen is that the conclusion would be $\mathcal{B}(\aleph_{\alpha + 1})\le 2^{\aleph_{\alpha + 1}}$ instead of $\mathcal{B}(\aleph_{\alpha + 1})< 2^{\aleph_{\alpha + 1}}$ (we may not be able to use the cofinality argument).
\end{cor}

\begin{defin}\label{k-linked}
    Let $\kappa$ be a regular cardinal. 
    \begin{enumerate}
        \item A poset $P$ is $\mathbf{stationary}$ $\mathbf{\kappa^{+}}$-$\mathbf{linked}$, if for every sequence of conditions $(p_{\gamma} | \gamma < \kappa^{+})$, there is a regressive function $f: \kappa^{+} \to \kappa^{+}$, such that for some club $C \subset \kappa^{+}$, for all $\gamma, \delta \in C$ with cofinality $\kappa, f(\gamma) = f(\delta)$ implies $p_{\gamma}$ and $p_{\delta}$ are compatible.
        \item A poset is $\mathbf{well}$-$\mathbf{met}$ if every two compatible conditions have a greatest lower bound.
     \item Let $\Gamma_{\kappa}$ be the class of $\kappa$-closed, stationary $\kappa^{+}$-linked, well met poset with greatest lower bounds.
     \item  For a regular $\kappa$, $\mathbf{GMA_{\kappa}}$ states that for every $P \in \Gamma_{\kappa}$, and for every collection of less than $2^{\kappa}$ many dense sets there is a filter for $P$ meeting them.
     \end{enumerate}
\end{defin}

\begin{defin}\label{SMP}
    For regular $\kappa$, $\mathbf{SMP_{n}(\kappa)}$ is the conjunction the following:
    \begin{itemize}
        \item $\kappa^{< \kappa} = \kappa$;
        \item for any $\Sigma_{n}$ statement $\phi$ with parameters in $H(2^{\kappa})$ and any $P \in \Gamma_{\kappa}$, if for all $P$-generic $G$, and all $\kappa$-closed, $\kappa^{+}$-c.c. posets $Q \in V[G]$, we have that $V[G] \models (1_{Q}$ forces $\phi$), then $\phi$ is true in $V$.
    \end{itemize}
    $\mathbf{SMP(\kappa)}$ is the statement that $SMP_{n}(\kappa)$ holds for all $n$.
\end{defin}

\underline{Fact:}(from \cite{15}) If $\kappa$ satisfies $\kappa^{< \kappa} = \kappa$, then a model of $SMP(\kappa)$ can be forced starting from a Mahlo cardinal $\theta > \kappa$.

The following Proposition is also due to \cite{15}.

\begin{prop}\label{SMP-GMA}
    Let $\kappa$ be a regular cardinal.
    \begin{enumerate}
        \item If for all $\tau < 2^{\kappa}, \tau^{< \kappa} < 2^{\kappa}$, then $SMP_{1}(\kappa)$ implies that $GMA_{\kappa}$.
        \item $SMP_{2}(\kappa)$ implies that $2^{\kappa}$ is weakly inaccessible and that for all $\tau < 2^{\kappa}, \tau^{< \kappa} < 2^{\kappa}$.
        \item $SMP_{2}(\kappa)$ implies that every $\Sigma_{1}^{1}$-subset of $\kappa^{\kappa}$ of cardinality $2^{\kappa}$ contains a perfect set.
    \end{enumerate}
\end{prop}

We are now ready to prove the following Theorem.

\begin{thm}\label{consistency}
    From a Mahlo cardinal, it is consistent with ZFC that $2^{\aleph_{\alpha}} < \mathcal{B}(\aleph_{\alpha + 1}) = 2^{\aleph_{\alpha + 1}}$, for every $\kappa < 2^{\aleph_{\alpha + 1}}$ there is an $\aleph_{\alpha + 1}$-Kurepa tree with at least $\kappa$-many maximal branches, but no $\aleph_{\alpha + 1}$-Kurepa tree has $2^{\aleph_{\alpha + 1}}$-many maximal branches.
\end{thm}

\begin{proof}

    Take $V$ to be a model of $SMP_{2}(\aleph_{\alpha + 1})$. By Proposition \ref{SMP-GMA}, in $V$ we have that (1) $GMA_{\aleph_{\alpha + 1}}$, 
 (2) $2^{\aleph_{\alpha}} = \aleph_{\alpha + 1}$, (3) $2^{\aleph_{\alpha + 1}}$ is weakly inaccessible and (4) every $\Sigma_{1}^{1}$-subset of $\aleph_{\alpha + 1}^{\aleph_{\alpha + 1}}$ of cardinality $2^{\aleph_{\alpha + 1}}$ contains a perfect set.

    We prove that this is the model we are looking for. 

    Assume $\aleph_{\alpha + 1} < \lambda < 2^{\aleph_{\alpha + 1}}$. In order to prove that there is an $\aleph_{\alpha + 1}$-Kurepa tree with at least $\lambda$-many maximal branches, we define $P$ to be the poset to add such a tree, i.e. we take the same poset as we did in \ref{forcing} but with $\lambda$ instead of $\aleph_{\omega_{\alpha + 1}}$. Briefly, we remind that the conditions are pairs $(t, f)$, where:
    \begin{itemize}
         \item $t$ is a tree of height $\beta + 1$ for some $\beta < \omega_{\alpha + 1}$ and levels of size at most $\aleph_{\alpha}$;
        \item $f$ is a function with $dom(f) \subset \lambda, |dom(f)| = \aleph_{\alpha}$ and $ran(f) = t_{\beta}$, where $t_{\beta}$ is the $\beta$th level of $t$.
    \end{itemize}

    The order is defined similarly as in the poset in Theorem \ref{forcing}.

    We prove that $P$ is in $\Gamma_{\aleph_{\alpha + 1}}$ and by $GMA_{\aleph_{\alpha + 1}}$ it follows that in the extension there is an $\aleph_{\alpha + 1}$-Kurepa tree with $\lambda$-many cofinal branches. 
    
    We will only check that $P$ is stationary $\aleph_{\alpha + 2}$-linked. Take a sequence of conditions $(p_{\gamma} | \gamma < \aleph_{\alpha + 2})$, where $p_{\gamma} = (t_{\gamma}, f_{\gamma})$. As in the proof of Theorem \ref{forcing}, we can assume that all $t_{\gamma}$'s are equal, since we have only $2^{\aleph_{\alpha}} = \aleph_{\alpha + 1}$-many possibilities for the first coordinate. 

    First we prove it for $\lambda = \aleph_{\alpha + 2}$. We define the regressive function $h(\gamma) = f_{\gamma}|_{\gamma}$ (the restriction of the corresponding function on $\gamma$). We now define $C \subseteq \omega_{\alpha + 2}$ to be a club, such that if $\gamma_{1}, \gamma_{2} \in C$ with $\gamma_{1} < \gamma_{2}$, then $\gamma_{2}$ should be above of $dom(f_{\gamma_{1}})$. We now observe that if $\gamma_{1}, \gamma_{2} \in C$ with $cf(\gamma_{1}) = cf(\gamma_{2}) = \aleph_{\alpha + 1}$ and $h(\gamma_{1}) = h(\gamma_{2})$ (i.e. $f_{\gamma_{1}}|_{\gamma_{1}} = f_{\gamma_{2}}|_{\gamma_{2}}$), then $f_{\gamma_{1}}$ and $f_{\gamma_{2}}$ are compatible. Indeed if we have that $f_{\gamma_{1}}|_{\gamma_{1}} = f_{\gamma_{2}}|_{\gamma_{2}}$, it means that the two functions agree on the part of $dom(f_{\gamma_{1}})$ below $\gamma_{1}$ and furthermore that $dom(f_{\gamma_{2}}) \cap [\gamma_{1}, \gamma_{2}) = \emptyset$. So, we have proven the desired property.

    In the case $\lambda > \aleph_{\alpha + 2}$, we can assign the $\aleph_{\alpha + 2}$-many domains of the functions we have in the sequence to sets below $\omega_{\alpha + 2}$ (preserving the order in which they appear and their intersections) and then work as in the previous case.
    
    Finally, we use the fourth property from those listed in the beginning of the proof, to prove that there are no $\aleph_{\alpha + 1}$-Kurepa trees with $2^{\aleph_{\alpha + 1}}$-many cofinal branches. Let $T$ be an $\aleph_{\alpha + 1}$-Kurepa tree. We prove that the set of branches $[T]$ is a closed set that does not contain a perfect set. In order to prove that, take $g: 2^{\omega_{\alpha + 1}} \to \omega_{\alpha + 1}^{\omega_{\alpha + 1}}$ to be a continuous injection with range contained in $[T]$. Then, we build sequences $(p_{s} | s \in 2^{< \omega_{\alpha}})$ and $(\beta_{\gamma} | \gamma < \omega_{\alpha})$, such that:
    
    \begin{itemize}
        \item if $s' \supset s$, then $p_{s} <_{T} p_{s'}$,
        \item if the order type of $s$ is $\gamma$, then $\beta_{\gamma} = domp_{s}$,
        \item for each $s, p_{s \frown 0} \neq p_{s \frown 1}$.
        \item at limit stages we take the unions
    \end{itemize}

    The idea here is to find a sequence $p_{s}$ of nodes of $T$, which will give us $2^{\omega_{\alpha}}$-many different branches before the maximal level of the tree. The $\beta_{\gamma}$'s are the level where the $p_s$'s split.

    We do this by induction on the order type of $s$, using continuity. Then let $\beta = sup_{\gamma} \beta_{\gamma}$ and for all $\eta \in 2^{\omega_{\alpha}}$, let $p_{\eta} = \bigcup_{\gamma} p_{\eta|_{\gamma}}$. But then if $\eta \neq \delta, p_{\eta} \neq p_{\delta}$. So, the $\beta$th level of the tree has $2^{\aleph_{\alpha}}$-many elements. Contradiction with $T$ being an $\aleph_{\alpha + 1}$-Kurepa tree (We may have the same contradiction, before we reach level $\beta$. More specifically, at some stage limit there is a possibility, that we take $> \aleph_{\alpha}$-many elements. In such case we stop the induction at that level.). So, $[T]$ does not contain a perfect set. By property (4) (from those we stated in the beginning of the proof), we have that $|[T]| < 2^{\aleph_{\alpha + 1}}$.
\end{proof}

\begin{cor}\label{results1}
    For every $\alpha$ countable ordinal, there exists an $\mathcal{L}_{\omega_1, \omega}$-sentence $\psi_{\alpha}$ that it is consistent with ZFC that:
    \begin{enumerate}
        \item $2^{\aleph_{\alpha}} \leq \aleph_{\omega_{\alpha + 1}}$ and $\psi_{\alpha}$ characterizes $\aleph_{\omega_{\alpha + 1}}$
        \item $2^{\aleph_{\alpha}} < 2^{\aleph_{\alpha + 1}}$, $2^{\aleph_{\alpha + 1}}$ is weakly inaccessible and $Spec(\psi_{\alpha}) = [\aleph_{0}, 2^{\aleph_{\alpha + 1}})$
        \item $2^{\aleph_{\alpha}} < 2^{\aleph_{\alpha + 1}}$, $2^{\aleph_{\alpha + 1}}$ is weakly inaccessible and $MM-Spec(\psi_{\alpha})$ is a cofinal subset of $[\aleph_{\alpha + 1}, 2^{\aleph_{\alpha + 1}})$
        
        If, in addition $\alpha > 0$ is finite, then it is also consistent that
        \item $2^{\aleph_{\alpha}} = \aleph_{\alpha + 1} < \aleph_{\omega_{\alpha + 1}}$ and $AP-Spec(\psi_{\alpha})$ contains the whole interval $[\aleph_{\alpha + 2}, \aleph_{\omega_{\alpha + 1}}]$ and possibly $\aleph_{\alpha + 1}$
        \item $2^{\aleph_{\alpha}} = \aleph_{\alpha + 1} < 2^{\aleph_{\alpha + 1}}$, $2^{\aleph_{\alpha + 1}}$ is weakly inaccessible and $AP-Spec(\psi_{\alpha})$ contains the whole interval $[\aleph_{\alpha + 2}, 2^{\aleph_{\alpha + 1}})$ and possibly $\aleph_{\alpha + 1}$
        
    \end{enumerate}
\end{cor}

\begin{proof}
    All of these results follow from Theorem \ref{spectrum}, Theorem \ref{maximal model spectrum}, Corollary \ref{results} and Theorem \ref{JEPAP} using the appropriate model of ZFC each time. For (1) use the model of Theorem \ref{forcing} and for (4) use again Theorem \ref{forcing} and Theorem \ref{JEPAP}. Finally for (2), (3) and (5) use the model of Theorem \ref{consistency} (here we assume the existence of a Mahlo cardinal). Note here that the amalgamation spectrum may not be equal to each interval, but it may contains more cardinals. This depends on the existence or not of $(\mu, \aleph_{\alpha}, \kappa)$-Kurepa trees, where $\mu \leq \aleph_{\alpha}$ (Theorem \ref{JEPAP}). 
\end{proof}

\underline{Observation:} Finally, since in Theorem \ref{forcing} we can replace $\aleph_{\omega_{\alpha + 1}}$ with any $\lambda \geq \aleph_{\alpha + 2}$ with $cf(\lambda) > \aleph_{\alpha}$, the same is true in the previous Corollary. 

\begin{cor}\label{nonabsolute}
The amalgamation property for $\mathcal{L}_{\omega_{1}, \omega}$-sentences is not absolute for cardinals $ 2^{\aleph_{\alpha}}<\kappa\leq 2^{\aleph_{\alpha + 1}}$. 
\end{cor}

\begin{proof}
    Let $2^{\aleph_{\alpha}} < \kappa \leq 2^{\aleph_{\alpha + 1}}$. From the comment right after Theorem \ref{forcing}, we can replace $\aleph_{\omega_{\alpha + 1}}$ by any $\lambda \geq \aleph_{\alpha +2}$. So, first we take that $\lambda$ to be $\geq \kappa$ (e.g. $\lambda = \kappa^{+}$). Then $\emptyset \neq [2^{\aleph_{\alpha}}, \kappa] \subseteq Spec(\psi_{\alpha})$. So by Theorem \ref{JEPAP}, $\kappa \in AP$-$Spec(\psi_{\alpha})$. On the other hand, let $\lambda$ be smaller than $ \kappa$, but $\lambda \geq \aleph_{\alpha + 2}$ (e.g. $\lambda = \aleph_{\alpha + 2}$). Then $\kappa \notin AP$-$Spec(\psi_{\alpha})$, since $\kappa \notin Spec(\psi_{\alpha})$ (Corollary \ref{results1} (1) combined with the last observation). 
 
    Observe here that if $\kappa = \aleph_{\alpha + 2} > 2^{\aleph_{\alpha}} = \aleph_{\alpha + 1}$, we cannot use Theorem \ref{forcing} to prove that $\kappa \notin AP$-$Spec(\psi_{\alpha})$. In this case, we construct a model with $2^{\aleph_{\alpha}} = \aleph_{\alpha + 1}$ and $Spec(\psi_{\alpha}) = [\aleph_{0}, \aleph_{\alpha + 1}]$, for example by killing all $\aleph_{\alpha + 1}$-Kurepa trees. Then  $2^{\aleph_{\alpha}} < \aleph_{\alpha + 2}$ and $\aleph_{\alpha + 2} \notin AP$-$Spec(\psi_{\alpha})$. The last statement holds, because if $|L_{\alpha + 1}| \leq \aleph_{\alpha}$, then our model cannot exceed the size of $2^{\aleph_{\alpha}}$ and we have noticed before that if $|L_{\alpha + 1}| = \aleph_{\alpha + 1}$, then we have either an $\aleph_{\alpha + 1}$-Kurepa tree or the tree has $\leq \aleph_{\alpha + 1}$-many cofinal branches and therefore the size of the model is $\leq \aleph_{\alpha + 1}$.
\end{proof}

We cannot conclude if the amalgamation property for $\mathcal{L}_{\omega_{1}, \omega}$-sentences is absolute or not for $\aleph_{1}$, since it is never included in the amalgamation spectrum.

\section*{Acknowledgements}

The author would like to thank Dima Sinapova for useful clarifications on the proofs of some of the Theorems from \cite{Kurepatrees} and also David Aspero that helped with some details of the proof of Theorem \ref{forcing}.

Furthermore, we want to thank Joseph Van Name and Monroe Eskew for answering two MathOverflow questions, \cite{q1} and \cite{q2} correspondingly.

\bibliographystyle{plain} 
\bibliography{bibliography.bib}

\end{document}